\documentclass[preprint,12pt]{elsarticle}
\usepackage{cmap} 
           \usepackage[T2A]{fontenc}   
           \usepackage[utf8]{inputenc}
           \usepackage[russian,english]{babel}

\usepackage{a4wide,amsfonts, amsmath, amscd}
\usepackage[psamsfonts]{amssymb}
\usepackage{hyperref}
\usepackage{amssymb}
 \usepackage{graphicx}
\usepackage{epsfig}

\usepackage{amsthm}

\newcommand{\sysn}{\left\{\begin{array}{rcl}}
\newcommand{\sysk}{\end{array}\right.}

\newcommand{\ra}{\rangle}
\newcommand{\la}{\langle}

\renewcommand{\le}{\leqslant}
\renewcommand{\ge}{\geqslant}

\def\si{\sigma}
\def\de{\delta}
\def\es{\varnothing}
\def\om{\omega}
\def\cF{\mathcal{F}}
\def\cQ{\mathcal{Q}}

\def\cP{\mathcal{P}}
\def\la{\lambda}
\def\ot{\mathbf{2}}
\def\ga{\gamma}
\def\al{\alpha}
\def\rarr{\Rightarrow}
\def\larr{\Leftarrow}
\def\lrarr{\Leftrightarrow}

\def\De{\Delta}
\def\Si{\Sigma}

\usepackage{xcolor}

\def\R{\mathbb{R}}
\def\Q{\mathbb{Q}}
\def\restr#1#2{#1 \upharpoonright #2}
\def\:{\colon}

\newtheorem{theorem}{Theorem}[section]
\newtheorem{lemma}[theorem]{Lemma}

\newtheorem{proposition}[theorem]{Proposition}

\newtheorem{corollary}[theorem]{Corollary}

\theoremstyle{definition}
\newtheorem{example}[theorem]{Example}
\newtheorem{definition}[theorem]{Definition}
\newtheorem{remark}[theorem]{Remark}
\newtheorem{question}[theorem]{Question}

\def\U{\mathcal{U}}

\journal{...}

\begin{document}

\title{Lindel\"{o}f-type properties of spaces of Baire-one functions}

\author{Alexander V. Osipov}
\address{Krasovskii Institute of Mathematics and Mechanics, \\ Ural Federal
 University, Yekaterinburg, Russia}
\ead{OAB@list.ru}

\author{Evgenii Reznichenko}
\address{Department of General Topology and Geometry, Mechanics and Mathematics Faculty, M.V.Lomonosov  Moscow State University, Leninskie  Gory 1, Moscow, 199991, Russia}
\ead{erezn@inbox.ru}

\begin{abstract} In this paper, we investigate Lindel\"{o}f-type properties of spaces of Baire-one functions equipped with the topology of pointwise convergence.
In particular, we investigate the properties of normality, extent and realcompactness for the space $B_1(X)$ of Baire-one real-valued functions. We also study when $B_1(X)$ is realcompact: we show that this is the case for every hereditarily Baire (in particular, every Polish) separable metrizable space $X$ and for every Luzin set $X$. The paper provides answers to two Teng's questions. The paper raises open questions.
\end{abstract}

\begin{keyword}   Baire-one function \sep  extent \sep realcompactness \sep Lindel\"{o}f property \sep normality \sep function spaces \sep $nCM$-spaces
\MSC[2010] 54C35 \sep 54C40
\end{keyword}

\maketitle


\section{Introduction}

Recall that a space $X$ is called {\it Lindel\"{o}f}  if $X$ is a regular space and every open cover of $X$ has a countable subcover of $X$. Investigating the Lindel\"{o}f property for a function space is a difficult problem. For instance, the question of whether space $C_p(X)$  is Lindelöf has remained unsolved for over 50 years.

In 1987, Pestryakov  investigated cardinal-valued functions on the space of Baire-one functions \cite{pes1,pes2} and, in particular, studied the Lindel\"{o}f property of the space $B_1(X)$. Various topological properties of spaces of Baire-one functions were studied in \cite{Gab1, GOR26, os}. More specifically, Pestryakov  investigated the following questions:

\medskip

{\it When $B_1(X)$ is Lindel\"{o}f?
In particular, is there an uncountable (separable metrizable) space $X$ for which the space $B_1(X)$ is Lindel\"{o}f?

\medskip

Is there an uncountable (separable metrizable) space $X$ for which the space $B_1(X)$ is normal and realcompact?}

\medskip

These questions remain open.

\medskip

Pestryakov obtained a number of necessary conditions for space $X$ under which space $B_1(X)$ is Lindel\"{o}f (see Proposition \ref{pr1}).

Note that the Lindel\"{o}f property for an arbitrary topological space $X$ implies that $X$ has countable extent, is normal, and is realcompact.

 Since Pestryakov's questions turned out to be quite difficult for us, we  study the properties that follow from  Lindel\"{o}f property of $B_1(X)$  such as normality, countable extent, and realcompactness.
 In investigating the question of the normality of space $B_1(X)$, we obtained answers to Teng's questions (\cite{Pea}, Problems L41 and L42 and in \cite{Problem}):

 (L41) {\it Let $X$ be a fortissimo space and $p$ the particular point of $X$. Let $Y=X\setminus \{p\}$. Is $C_p(Y|X)$ normal?}

 (L42) {\it With $X$ and $Y$ as in L41, is $C_p(Y|X)$ homeomorphic to the $\Sigma$-product of $|X|$-many real lines?}

In Example \ref{exa:norm:2} and Remark \ref{r:teng}, we show that the answers to Teng's problems  are negative.


In Theorem \ref{t:pr:1} we obtain some strengthening of the properties obtained by Pestryakov in \ref{pr1}. Namely, if a space $B_1(X)$ is Lindel\"{o}f for some space $X$, then
$X$ is a perfect zero-dimensional space, any closed subset of $X$ is a $CZ$-set, $s(X)=\aleph_0$, and any continuous image of  $X$ is $k$-scattered.

In Proposition \ref{c:enba:2} we show that if $B_1(X)$ is normal then $B_1(X)$ has countable extent; together with Lemma \ref{l:cantor:2} this implies that $X$ is an $nCM$-space. In Theorem \ref{t:en:1} we prove a similar statement for spaces of (functionally) Baire-one maps with values in a separable metrizable space: if such a space is weakly normal, then it has countable extent.
We also investigate when $B_1(X)$ is realcompact. We show that $B_1(X)$ is realcompact for every hereditarily Baire (in particular, every Polish) separable metrizable space $X$, for every Luzin set $X\subset\R$, and for every $\si$-compact metrizable space $X$ (Corollary \ref{t:rc:1} and the remark after it). On the other hand, if $X$ is a $\la$-set which is not a $Q$-set, then $B_1(X)$ is not realcompact (Proposition \ref{p:lambdaY}); such sets exist under CH or MA, but we do not know whether they exist in ZFC. We also do not know whether $B_1(X)$ is realcompact for every Suslin (analytic) separable metrizable space $X$ (Question \ref{q:analytic}).

Interesting results have been obtained for classical topological spaces.

(1) If $X\subset \mathbb{R}$ is a Luzin set then the space $B_1(X)$ is realcompact (Corollary \ref{t:rc:1}).

(2) If $K$ is a Cantor discontinuum then the extent of the space of all indicator Baire-one functions $f\colon K\to\{0,1\}$ is equal to the continuum (Theorem \ref{t:sp:n:1}).

(3) If $X$ is a Bernstein set then $B_1(X)$ is not normal and the extent of $B_1(X)$ is equal to the continuum (Theorem \ref{t:bern}). Thus, for any  Bernstein set $X\subset \R$ the space $B_1(X)$ is realcompact, but not normal (Corollary \ref{c:bern}).

Pestryakov proved that if $B_1(X)$ is Lindel\"{o}f for some space $X$ then {\it for any continuous function $f:X\rightarrow \mathbb{R}$ the space $f(X)$ contains no uncountable compact subspace.}
We show that this property of the space $X$  is equivalent to being $nCM$-space (Remark \ref{r:cantor}). Thus, we separately investigate $nCM$-spaces (see Section \ref{sec:ncm}) as a necessary condition under which $B_1(X)$ is Lindel\"{o}f.

\section{Preliminaries}

\subsection{Main definitions}

Throughout this paper,  all spaces are assumed to be Tychonoff.

A real-valued function $f$ on a space $X$ is a {\it Baire-one
function} (or a {\it function of the first Baire class}) if $f$ is
a pointwise limit of a sequence of continuous functions on $X$.
Let $B_1(X)$ denote the space of all Baire-one real-valued
functions on a space
 $X$ with the topology of pointwise convergence.
 It is well known that the space $B_1(X)$ forms a vector space which
contains the space of all continuous functions $C(X)$ and which is
closed with respect to the uniform convergence. Moreover, $f\cdot
g$ and $\max(f,g)$ are Baire-one functions whenever $f,g\in
B_1(X)$.

For $f\in\R^X$, a finite set $K\subset X$ and $\varepsilon>0$ we put
$[f,K,\varepsilon]=\{g\in \R^X: |g(x)-f(x)|<\varepsilon$ for $x\in K\}$ and $\langle f,x,\varepsilon\ra=[f,\{x\},\varepsilon]$.

 We recall that a subset of $X$ that is the
 complete preimage of zero for a certain function from~$C(X)$ is called a zero-set.
A subset $O\subseteq X$  is called  a cozero-set (or functionally
open) of $X$ if $X\setminus O$ is a zero-set of $X$. It is easy to
check that zero sets are preserved by finite unions and countable
intersections. Hence cozero sets are preserved by finite
intersections and countable unions.

Countable unions of zero sets
will be denoted by $Zer_{\sigma}$ (or $Zer_{\sigma}(X)$),
countable intersection of cozero sets by $Coz_{\delta}$ (or
$Coz_{\delta}(X)$). It is easy to check that $Zer_{\sigma}$-sets
are preserved by countable unions and finite intersections. Note
that if $Y\subset X$ and $A\in Coz_{\delta}(X)$ then $A\cap Y\in
Coz_{\delta}(Y)$. Note also that any zero set is $Coz_{\delta}$
and any cozero-set is $Zer_{\sigma}$.
If $g\colon X\to Y$ is continuous, then $g^{-1}(A)\in Zer_\si(X)$ for $A\in Zer_\si(Y)$ and $g^{-1}(A)\in Coz_\de(X)$ for $A\in Coz_\de(Y)$.

Let $Y$ be a space. Denote by $B_1(X,Y)$ the set of all mappings $f$ from $X$ into $Y$ such that $f$ is the pointwise limit of a sequence of functions from $C(X,Y)$.

Denote by $K_1(X,Y)$ the set of all mappings $f$ from $X$ into $Y$ such that $f^{-1}(U)\in Zer_{\sigma}(X)$ for any cozero set $U\subset Y$.

It is well known that $f$ is
of the first Baire class if and only if $f^{-1}(U)\in
Zer_{\sigma}$ for every open $U\subseteq \mathbb{R}$ (see Exercise
3.A.1 in \protect\cite{lmz1}).
It follows that $B_1(X):=B_1(X,\R)=K_1(X,\R)$.

A subset $A$ of a space $X$ is called a {\it $CZ$-set}, if $A\in Coz_{\delta}(X)\cap Zer_{\sigma}(X)$.
Any zero set and cozero set is a $CZ$-set. Finite unions and finite intersections of $CZ$-sets, and complements of $CZ$-sets, are $CZ$-sets.

Denote by $\ot=\{0,1\}$. A function $f: X\rightarrow \ot$  belongs to $K_1(X,\ot)$ if and only if $f^{-1}(0)$ and $f^{-1}(1)$ are $CZ$-sets. Note that $K_1(X,\ot)=\{ f\in B_1(X): f(X)\subset\ot\}$; in particular, $K_1(X,\ot)$ is a closed subspace of $B_1(X)$.

Note also that $B_1(X, \ot)\subseteq K_1(X,\ot)$, but there are spaces $X$ such that $B_1(X, \ot)$ is trivial. For instance, if $X$ is a connected space then $B_1(X, \ot)$ consists only of the two constant functions.

\subsection{Auxiliary lemmas}

\begin{lemma}[gluing]\label{l:glue}
Let $M$ be a $CZ$-set in $X$, $Y$ a space and $f,g\in K_1(X,Y)$. Then the function $h$, $h\restriction M=f\restriction M$, $h\restriction (X\setminus M)=g\restriction (X\setminus M)$, belongs to $K_1(X,Y)$.
\end{lemma}
\begin{proof}
For a cozero set $U\subset Y$, $h^{-1}(U)=(f^{-1}(U)\cap M)\cup(g^{-1}(U)\cap (X\setminus M))$. Since $M$ and $X\setminus M$ are $Zer_\si$-sets and $Zer_\si$ is closed under finite intersections and countable unions, $h^{-1}(U)\in Zer_\si(X)$.
\end{proof}

\begin{lemma}[factorization]\label{l:fact}
Let $Y$ be a separable metrizable space and $f\in K_1(X,Y)$. Then there are a continuous map $\varphi\colon X\to [0,1]^\om$ and a map $h\colon \varphi(X)\to Y$ such that $f=h\circ\varphi$ and $h^{-1}(U)\in Zer_\si(\varphi(X))$ for every open $U\subset Y$.
In particular, for every nonempty zero set $F\subset X$ there is a nonempty zero set $G\subset F$ such that $f\restriction G$ is constant.
\end{lemma}
\begin{proof}
Let $\{U_n:n\in\om\}$ be a base of $Y$. Every $U_n$ is a cozero set, hence $f^{-1}(U_n)=\bigcup_k \psi_{n,k}^{-1}(0)$ for some continuous $\psi_{n,k}\colon X\to[0,1]$. Let $\varphi=(\psi_{n,k})_{n,k\in\om}\colon X\to [0,1]^{\om\times\om}\cong[0,1]^\om$. If $\varphi(x)=\varphi(x')$, then $x\in f^{-1}(U_n)\lrarr x'\in f^{-1}(U_n)$ for all $n$, hence $f(x)=f(x')$. Thus $h=f\circ\varphi^{-1}$ is well defined on $\varphi(X)$ and
$h^{-1}(U_n)=\varphi(f^{-1}(U_n))=\bigcup_k\{t\in\varphi(X): t_{n,k}=0\}\in Zer_\si(\varphi(X))$. Every open $U\subset Y$ is a union of countably many $U_n$, so $h^{-1}(U)\in Zer_\si(\varphi(X))$.

If $F$ is a nonempty zero set and $x\in F$, then $G=F\cap\varphi^{-1}(\varphi(x))$ is a nonempty zero set and $f$ is constant on $G$.
\end{proof}

For $x\neq y$ in $\ot^\om$ let $\De(x,y)=\min\{n\in\om: x(n)\neq y(n)\}$.

\begin{lemma}\label{l:cantor:1}
For $y\in\ot^\om$ define $f_y\colon \ot^\om\to\om$ by $f_y(y)=0$ and $f_y(x)=\De(x,y)+1$ for $x\neq y$. Then $D=\{f_y:y\in\ot^\om\}\subset B_1(\ot^\om)$ is a closed discrete subset of $\R^{\ot^\om}$ of cardinality~$\mathfrak c$.
\end{lemma}
\begin{proof}
Every set $\{x: \De(x,y)=n\}$ is clopen and $\{y\}$ is closed, hence $f_y^{-1}(U)\in Zer_\si(\ot^\om)$ for every $U\subset\R$, so $f_y\in B_1(\ot^\om)$.
Since $f_y(y)=0$ and $f_z(y)\geq 1$ for $z\neq y$, the neighbourhood $\la f_y,y,1\ra$ meets $D$ only in $f_y$; thus $D$ is discrete and $|D|=\mathfrak c$.
Let $(f_{y_\la})_\la$ be a net converging in $\R^{\ot^\om}$ to $g$. Passing to a subnet we may assume that $y_\la\to y^*$. If $y_\la\neq y^*$ for all $\la$ of a subnet, then $f_{y_\la}(y^*)=\De(y^*,y_\la)+1\to\infty$, which is impossible. Hence $y_\la=y^*$ for cofinally many $\la$, and $g=f_{y^*}\in D$.
\end{proof}

\begin{lemma}\label{l:cantor:2}
Assume that there is $f\in C(X)$ such that $f(X)$ contains a subset homeomorphic to the Cantor set. Then $B_1(X)$ contains a closed discrete subset of cardinality $\mathfrak c$; in particular $e(B_1(X))\geq\mathfrak c$.
\end{lemma}
\begin{proof}
Let $H\subset f(X)$, $c\colon H\to\ot^\om$ be a homeomorphism, $Z=f^{-1}(H)$ and $p=c\circ (f\restriction Z)\colon Z\to\ot^\om$; $p$ is a continuous surjection and $Z$ is a zero set of $X$. Define $T\colon\R^{\ot^\om}\to\R^X$ by $T(u)\restriction Z=u\circ p$ and $T(u)\restriction(X\setminus Z)=0$. Then $T$ is a homeomorphic embedding and $T(\R^{\ot^\om})=\{v\in\R^X: v$ is constant on the fibres of $p$ and $v(X\setminus Z)=\{0\}\}$ is closed in $\R^X$.

If $Q\subset\ot^\om$ is closed, then $c^{-1}(Q)$ is a closed subset of $\R$, so $p^{-1}(Q)=f^{-1}(c^{-1}(Q))$ is a zero set of $X$. Hence $p^{-1}(A)\in Zer_\si(X)$ for $A\in Zer_\si(\ot^\om)$, and, since $X\setminus Z$ is a cozero set, $T(u)^{-1}(U)\in Zer_\si(X)$ for $u\in B_1(\ot^\om)$ and open $U\subset\R$. So $T(B_1(\ot^\om))\subset B_1(X)$. By Lemma \ref{l:cantor:1}, $T(D)$ is a closed discrete subset of $\R^X$ of cardinality $\mathfrak c$ contained in $B_1(X)$.
\end{proof}

\begin{remark}\label{r:cantor}
For a space $X$ the following are equivalent: (a) $X$ can be continuously mapped onto $[0,1]$; (b) some $f\in C(X)$ has $f(X)\supset H$ with $H$ homeomorphic to the Cantor set; (c) some $f\in C(X)$ has uncountable compact $K\subset f(X)$. Indeed, (a)$\rarr$(c) is trivial, (c)$\rarr$(b) since an uncountable compact subset of $\R$ contains a Cantor set, and (b)$\rarr$(a): take a continuous $g\colon\R\to[0,1]$ with $g(H)=[0,1]$ (the Cantor function extended to $\R$). Spaces not satisfying (a) are called $nCM$-spaces (Section \ref{sec:ncm}).
\end{remark}

\section{Pestryakov's and new results}

In \cite{pes1,pes2} (see also \cite{os}), Pestryakov proved the following results.

\begin{proposition}[Pestryakov]\label{pr1} If a space $B_1(X)$ is Lindel\"{o}f for some space $X$, then
\begin{enumerate}
\item
any closed subset of $X$ is a $Coz_{\delta}$-set in $X$, in particular, $\psi(X)=\aleph_0$;
\item
any discrete family of open subsets of $X$ is countable;
\item
for any continuous function $f:X\rightarrow \mathbb{R}$ the space $f(X)$ contains no uncountable compact subspace;
\item
if $X$ is uncountable then $B_1(X)\neq B_2(X)$;
\item
if $X$ is functionally countable then $X$ is countable;
\item
if $X$ is pseudocompact then $X$ is countable.
\end{enumerate}
\end{proposition}

\begin{proof}
(1). Assume that $F$ is a closed subset of $X$, but it is not $Coz_{\delta}$-set in $X$. Then for any $f\in B_1^F\triangleq\{\varphi\in B_1(X): \varphi\upharpoonright F=0\}$ there is $x_f\in X\setminus F$ such that $|f(x_f)|<\frac{1}{2}$ (otherwise $f^{-1}(0)=F$ is a $Coz_{\delta}$-set in $X$). The family $\{\langle f, x_f, \frac{1}{2} \rangle: f\in B_1^F\}$ is an open cover of $B_1^F$. Since $B_1^F$ is a closed subset of $B_1(X)$, it is Lindel\"{o}f. Hence, there is a countable subcover $\gamma$ of $B_1^F$. Let $S=\{x_f: \langle f, x_f, \frac{1}{2} \rangle\in \gamma\}=\{s_i: i\in\omega\}$.

Consider $\varphi_i\in C(X,[0,1])$ such that $\varphi_i(F)\subset \{0\}$ and $\varphi_i(s_i)=1$ for every $i\in\omega$. Let $h_n=\min\{1, \sum_{i\le n} \varphi_i\}$ for every $n\in \omega$. Note that $h_n\in C(X)$ for every $n\in \omega$. Since the sequence $(h_n)_{n\in\omega}$ is monotone bounded, there is $g\in B_1(X)$ such that $g=\lim_{n\rightarrow \infty} h_n$. Note that $g(F)\subseteq \{0\}$ and $g(S)\subseteq \{1\}$.
Thus, $g\in B_1^F$ and $g\not\in \bigcup \gamma$, a contradiction.

(2). Let $\{U_\al:\al<\om_1\}$ be a discrete family of nonempty open sets, $x_\al\in U_\al$. Let $g\colon \R^X\to \R^{\om_1}$, $g(\varphi)=(\varphi(x_\al))_{\al<\om_1}$; $g$ is continuous. For $r\in\R^{\om_1}$ take $\varphi_\al\in C(X,[0,1])$ with $\varphi_\al(x_\al)=1$ and $\varphi_\al(X\setminus U_\al)=\{0\}$; the function $\varphi=\sum_\al r_\al\varphi_\al$ is continuous since the family $\{U_\al\}$ is discrete, and $g(\varphi)=r$. Hence $g(C(X))=g(B_1(X))=\R^{\om_1}$. The space $\R^{\om_1}$ is not normal (A.H. Stone, see \cite{Eng}), hence it is not Lindel\"of, and neither is $B_1(X)$.

(3). Assume that $K\subset f(X)$ where $f\in C(X)$ and $K$ is an uncountable compact set. Then $K$ contains a copy of the Cantor set and, by Lemma \ref{l:cantor:2}, $B_1(X)$ contains an uncountable closed discrete subset. This contradicts the Lindel\"of property.

(4). If $B_1(X)=B_2(X)$ then $B_2(X)$ is Lindel\"{o}f. Since a Lindel\"{o}f space is realcompact, $B_2(X)$ is realcompact. It follows that $B_2(X)=\mathbb{R}^X$ (see  Corollary 2.4 in \cite{os1}).  Thus $\R^X$ is Lindel\"of, hence normal, and therefore $X$ is countable (the space $\R^{\om_1}$ is not normal).

(5). We claim that if $X$ is functionally countable then $B_1(X)=B_2(X)$. For any Baire subset $A$ of $X$ there are $g\in C(X,\mathbb{R}^{\omega})$ and a Baire subset $H$ of $\mathbb{R}^{\omega}$ such that $g^{-1}(H)=A$. Since $X$ is functionally countable, $g(X)\subset\prod_n g_n(X)$ is a product of countable sets, so $g(X)$ is zero-dimensional and separable metrizable and there is a homeomorphic embedding $h\colon g(X)\to\R$. Then $h\circ g\in C(X)$ and $h\circ g(X)$ is countable. Hence, $g(X)$ and $H\cap g(X)$ are countable. Since every countable subset of $\mathbb{R}^{\omega}$ is a $Zer_{\sigma}$-set, $A$ is a $Zer_{\sigma}$-set in $X$. Hence every Baire-two function is Baire-one. By (4), $X$ is countable.

(6). Since a separable metrizable continuous image of a pseudocompact space is compact, $X$ is functionally countable, by (3), and, by (5), $X$ is countable.
\end{proof}

\begin{remark}\label{r:gor}
By \cite[Theorem 3.4]{GOR26}, $B_1(X)=B_2(X)$ if and only if $X$ is a {\it $CZ$-space}, i.e. every $Zer_\si$-subset of $X$ is a $Coz_\de$-set; by \cite[Theorem 2.8]{GOR26}, every functionally countable space is a $CZ$-space. This gives another proof of items (4) and (5) of Proposition \ref{pr1}. Moreover, if $B_1(X)$ is Lindel\"of and $X$ is uncountable, then $X$ is not a $CZ$-space, although every closed subset of $X$ is a $CZ$-set by Theorem \ref{t:pr:1}(2) below. For a separable metrizable $X$ this means that $X$ is not a $\sigma$-set.
\end{remark}

Recall that a  space $X$ is called {\em $k$-scattered} if every compact  subspace of $X$ is scattered.
The {\it spread} $s(X)$ is defined as follows: $s(X)=\sup\{|D|: D$ is a discrete subset of $X\}$.
Recall that a topological space is said to be {\it perfect} if every its closed subset is
$G_{\delta}$.

\medskip

The following statement is a refinement of Pestryakov's results.

\begin{theorem}\label{t:pr:1} If a space $B_1(X)$ is Lindel\"{o}f for some space $X$, then
\begin{enumerate}
\item
$X$ is zero-dimensional;
\item
$X$ is perfect and any closed subset of $X$ is a $CZ$-set;
\item
 $s(X)=\aleph_0$;
\item
any continuous image of  $X$ is $k$-scattered.
\end{enumerate}
\end{theorem}
\begin{proof}
(1).  Since $X$ is Tychonoff, for any point $x$ and a closed non-empty subset $F$ of $X$ ($x\notin F$) there is a continuous function $f: X\rightarrow [0,1]$  such that $f(x)=0$ and $f(F)=\{1\}$. By Proposition \ref{pr1} (3),
$f(X)\neq [0,1]$. Let $a\in [0,1]\setminus f(X)$. Then $W=f^{-1}([0,a))$ is a clopen set such that $x\in W\subset (X\setminus F)$.

(2). Let $F\subset X$ be a closed set. Let $\cQ=\{ Q\subset X: F\subset Q$ and $Q$ is a zero set $\}$,
$\cF_Q=\{ f\in K_1(X,\ot): f(F)=\{0\}$ and $f(X\setminus Q)=\{1\}\}$ for $Q\in \cQ$. Then  $\cF_{Q}$ is a closed set in $B_1(X)$ for any $Q\in\cQ$. If $Q_n\in\cQ$, $n\in\om$, then $Q=\bigcap_n Q_n\in\cQ$ and the indicator function of $X\setminus Q$ belongs to $\bigcap_n\cF_{Q_n}$. Thus the family $\{\cF_Q:Q\in\cQ\}$ has the countable intersection property. Since $B_1(X)$ is Lindel\"{o}f,  $\cF=\bigcap \{\cF_Q:Q\in\cQ\}\neq\es$. Let $f\in \cF$. Then $f(F)=\{0\}$  and $f(X\setminus F)=\{1\}$ (for $x\notin F$ there is $Q\in\cQ$ with $x\notin Q$, since $X$ is Tychonoff). It follows that $F$ is a $CZ$-set.

(3). By (2), closed sets, and hence open sets, are $CZ$-sets in $X$. Let $D\subset X$ be discrete. Then $D=\overline D\cap U$ for some open $U$, so $D$ is a $CZ$-set. For $A\subset D$ we have $A=\overline A\cap D$, which is a $CZ$-set by (2). Define $T\colon\R^D\to\R^X$, $T(u)\restriction D=u$, $T(u)\restriction(X\setminus D)=0$. For open $U\subset\R$ the set $T(u)^{-1}(U)$ is either $u^{-1}(U)$ or $u^{-1}(U)\cup (X\setminus D)$, a $Zer_\si$-set. Hence $T$ is a homeomorphism of $\R^D$ onto the closed subspace $\{v\in B_1(X): v(X\setminus D)=\{0\}\}$ of $B_1(X)$. Thus $\R^D$ is Lindel\"of and $D$ is countable.

(4). Let $g\colon X\to Y$ be a continuous surjection and $K\subset Y$ a non-scattered compact set. Then there is a continuous map of $K$ onto $[0,1]$; it extends to a continuous map $q\colon Y\to[0,1]$ ($K$ is $C^*$-embedded in $\beta Y$, hence in $Y$). Then $q\circ g$ maps $X$ onto $[0,1]$, contrary to Proposition \ref{pr1}(3).

\end{proof}

A closed $CZ$-set need not be a zero set, even in a separable first countable space; for normal spaces, however, every closed $CZ$-set is a zero set, since in a normal space every closed $G_\de$-set is a zero set.

\begin{example}\label{ex:niem}
Let $L$ be the Niemytzki plane, $P=\R\times\{0\}$ and $A=\Q\times\{0\}$. Then $A$ is a closed $CZ$-set in $L$ which is not a zero set.
\end{example}
\begin{proof}
The set $A$ is closed, since $P$ is a closed discrete subspace of $L$. For $a\in P$ and $r>0$ let $D(a,r)$ be the open disc of radius $r$ in the upper half-plane tangent to $P$ at $a$; there is a continuous $f_{a,r}\colon L\to[0,1]$ with $f_{a,r}^{-1}((0,1])=D(a,r)\cup\{a\}$.

$A\in Zer_\si(L)$: $L$ is first countable and Tychonoff, so every point is a zero set, and $A$ is countable.

$A\in Coz_\de(L)$: let $A=\{a_k:k\in\om\}$ and $U_n=\{\sum_k 2^{-k}f_{a_k,2^{-k}/n}>0\}=\bigcup_k\big(D(a_k,2^{-k}/n)\cup\{a_k\}\big)$. Each $U_n$ is a cozero set, $U_n\cap P=A$, and every point $(x,y)\in U_n$ with $y>0$ satisfies $y<2/n$. Hence $\bigcap_n U_n=A$.

$A$ is not a zero set: suppose $A=g^{-1}(0)$, $g\in C(L)$. For $x\in\R\setminus\Q$ there is $k$ such that $|g|>1/k$ on $D((x,0),1/k)\cup\{(x,0)\}$; let $E_k$ be the set of such $x$. Since $\R\setminus\Q=\bigcup_k E_k$ is a Baire space (in the Euclidean topology), some $E_k$ is dense in an open interval $I$. Take $a\in A$ with $a=(q,0)$, $q\in I$, and $0<s\le 1/k$ with $|g|<1/k$ on $D(a,s)\cup\{a\}$. Take $x\in E_k$ with $d=|x-q|<s$. The discs $D((x,0),R)$, $R=1/k$, and $D(a,s)$ intersect, since the distance between their centres satisfies $d^2+(R-s)^2<(R+s)^2\iff d^2<4Rs$, which holds as $d<s\le R$. At a common point we get $|g|>1/k$ and $|g|<1/k$, a contradiction.
\end{proof}

\begin{question}\label{q:cz:normal}
Assume that $X$ is a space such that any closed subset of $X$ is a $CZ$-set. Is $X$ a normal space ?
\end{question}

\begin{remark}
If $X$ is separable and every closed subset of $X$ is a $CZ$-set, then $2^{|D|}\le\mathfrak c$ for every closed discrete $D\subset X$. Indeed, $A\mapsto\overline A$ is an injection of $\mathcal P(D)$ into the family of $CZ$-sets, and $|C(X)|\le\mathfrak c$. In particular, the Niemytzki plane and the Mr\'owka spaces $\Psi(\mathcal A)$ with $|\mathcal A|=\mathfrak c$ do not satisfy the assumption of Question \ref{q:cz:normal}.
\end{remark}

\begin{remark}
A space is called a {\it $Q_f$-space} if every its subset is a $CZ$-set \cite{GOR26}; equivalently, $B_1(X)=\R^X$. Every closed subset of a $Q_f$-space is a $CZ$-set. Hence a positive answer to Question \ref{q:cz:normal} would give a negative answer to Problem 2.18 of \cite{GOR26} (is there a non-normal $Q_f$-space?), and a non-normal $Q_f$-space would answer Question \ref{q:cz:normal} in the negative.
\end{remark}

\begin{question}\label{q:lind} Assume that $B_1(X)$ is Lindel\"{o}f. Which of the following statements are true?
\begin{enumerate}
\item
$X$ is normal;
\item
$X$ is hereditarily Lindel\"{o}f;
\item
$X$ is hereditarily separable;
\item
$X$ is submetrizable;
\item
$X$ is cosmic.
\end{enumerate}
\end{question}

\section{Realcompactness}\label{sec:sp}

Recall \cite{arch} that a subspace $Z$ of a space $T$ is {\it $\aleph_0$-placed} ({\it $G_\de$-closed}) in $T$ if for every $p\in T\setminus Z$ there is a $G_\de$-set $P$ of $T$ with $p\in P\subset T\setminus Z$. The following fact is well known; it follows from Propositions II.4.9, II.4.10 and II.4.13 (the latter is essentially due to Chigogidze) in \cite{arch}, see also \cite{Arh83}.

\begin{proposition}\label{p:dense}
Let $X$ be a set, $Y$ a separable metrizable space and $Z$ a dense subspace of $Y^X$. Then $Z$ is realcompact if and only if $Z$ is $\aleph_0$-placed in $Y^X$, i.e., for every $f\in Y^X\setminus Z$ there is a countable $S\subset X$ such that $h\restriction S\ne f\restriction S$ for all $h\in Z$.
\end{proposition}
\begin{proof}
Every $G_\de$-subset of $Y^X$ containing $f$ contains a set $W_{f,S}=\{g\in Y^X: g\restriction S=f\restriction S\}$ with $S$ countable, and $W_{f,S}$ is a zero set of $Y^X$; this gives the second form of the condition.

($\Leftarrow$) $Z=\bigcap\{Y^X\setminus W_{f,S_f}: f\in Y^X\setminus Z\}$ is an intersection of cozero subsets of the realcompact space $Y^X$, so it is realcompact \cite{Eng}.

($\Rightarrow$) Let $f\in Y^X\setminus Z$ be such that every $W_{f,S}$ meets $Z$. By the factorization lemma \cite[0.2.3]{arch}, every continuous $\varphi\colon Z\to\R$ has the form $\varphi=g\circ p_S\restriction Z$ for a countable $S\subset X$ and a continuous function $g$ on $p_S(Z)$, where $p_S\colon Y^X\to Y^S$ is the projection. Since $W_{f,S}\cap Z\ne\es$, we have $p_S(f)\in p_S(Z)$, and $g\circ p_S$ is a continuous extension of $\varphi$ to $Z\cup\{f\}$. Thus $Z$ is $C$-embedded in $Z\cup\{f\}$. A realcompact space which is $C$-embedded in a Tychonoff space is closed in it \cite{Eng}, so $Z$ is not realcompact.
\end{proof}

\begin{remark}\label{r:os1}
For $Z=C_p(X)\subset\R^X$ Proposition \ref{p:dense} yields the theorem of Arhangel'skii \cite{Arh83}: $C_p(X)$ is realcompact iff every strictly $\aleph_0$-continuous function on $X$ is continuous; see also \cite{Usp83}. For $Z=B_1(X,Y)$ it yields \cite[Corollary 2.1]{os1}: $B_1(X,Y)$ is realcompact iff every $f\colon X\to Y$ which coincides on each countable $S\subset X$ with some $h\in B_1(X,Y)$ belongs to $B_1(X,Y)$.
\end{remark}

Next, we investigate realcompactness of $K_1(X,\ot)$ and $B_1(X)$.

\subsection{Realcompactness of $K_1(X,\ot)$}
\label{sec:realcompact2}

\begin{proposition}\label{p:en:1}
Let $X$ be an uncountable functionally Hausdorff space. Then there exists a disjoint uncountable family of zero-subsets of $X$.
\end{proposition}
\begin{proof}
If every point of $X$ is a zero set, then $\{\{x\}:x\in X\}$ is the required family.
Suppose that some point $x\in X$ is not a zero set in $X$. We construct by induction zero sets $F_\al\ni x$ and $S_{\al+1}$, $\al<\om_1$. Let $F_0=X$. If $F_\al$ is constructed, then $F_\al\neq\{x\}$, since $\{x\}$ is not a zero set; take $y_\al\in F_\al\setminus\{x\}$ and a continuous $f_\al\colon X\to[0,1]$ with $f_\al(x)=0$, $f_\al(y_\al)=1$, and put $F_{\al+1}=F_{\al}\cap f_{\al}^{-1}(0)$ and $S_{\al+1}=F_{\al}\cap f_{\al}^{-1}(1)\ni y_\al$.
If $\al$ is limit, then $F_{\al}=\bigcap \{F_{\gamma}: \gamma<\alpha\}$ is a zero set as a countable intersection of zero sets.
If $\beta<\ga$, then $S_{\ga+1}\subset F_\ga\subset F_{\beta+1}$ and $S_{\beta+1}\cap F_{\beta+1}=\es$.
Thus, $\{S_{\alpha+1}: \alpha<\omega_1\}$ is a disjoint uncountable family of nonempty zero-subsets of $X$.
\end{proof}

\begin{lemma} \label{l:sp:1}
Let $X$ be a space and let $\mathcal{F}$ be a countable cover of $X$ consisting of zero-sets of $X$.
Then there exists a countable partition $\mathcal{P}$ of $X$ consisting of $CZ$-sets of $X$ such that $\mathcal{P}$ is a refinement of $\cF$.
\end{lemma}
\begin{proof}
Let $\cF=\{F_n:n\in\om\}$. Consider $P_n=F_n\setminus \bigcup_{i<n}F_i$ for every $n\in\om$; $P_n$ is the intersection of a zero set and a cozero set, hence a $CZ$-set. Then $\cP=\{P_n:n\in\om\}$.
\end{proof}

Recall that subsets $A$ and $B$ of a space $X$ are  {\em $G_\de$-separated}, if there are $G_\de$ sets $P$ and $R$  such that $P\cap R=\es$, $A\subset P$ and $B\subset R$.
We call $A$ and $B$ {\em $CZ$-separated} if there is a $CZ$-set $Q$ such that  $A\subset Q$,  $B\subset X\setminus Q$.

\begin{lemma} \label{lem2}
Lindel\"{o}f  disjoint subsets $A$ and $B$ of a space $X$ are $G_\de$-separated  if and only if $A$ and $B$ are $CZ$-separated.
\end{lemma}
\begin{proof}
$(\larr)$ Let $P=Q$ and $R=X\setminus Q$. Then $P$ and $R$ are $G_\de$,  $P\cap R=\es$, $A\subset P$ and $B\subset R$.

$(\rarr)$ Since $A$ and $B$ are $G_\de$-separated in $X$, there are $G_\de$ sets $P$ and $R$  such that $P\cap R=\es$, $A\subset P$ and $B\subset R$.
There are sequences $(U_n)_n$ and $(V_n)_n$ of open sets of $X$ such that $P=\bigcap_n U_n$, $R=\bigcap_n V_n$.
Let $n\in\om$. Since $A$ is Lindel\"{o}f  there is a cozero set $U_n'$ such that $A\subset U_n'\subset U_n$. Similarly, there is a cozero set $V_n'$ such that $B\subset V_n'\subset V_n$.

Let $\cF=\{X\setminus U_n', X\setminus V_n' : n\in\om  \}$. The countable family $\cF$ of zero-sets is a cover of $X$ because $\bigcap_n U'_n\cap\bigcap_n V'_n\subset P\cap R=\es$. For the family $\cF$ holds: if $F\in \cF$, then  $A\cap F=\es$ or $B\cap F=\es$.
By Lemma \ref{l:sp:1}, there is a countable partition $\cP$ of $X$ consisting of $CZ$-sets such that $\mathcal{P}$ is a refinement of $\cF$. Then $A\cap P=\es$ or $B\cap P=\es$ for $P\in\cP$. Let $Q=\bigcup\{P\in\cP: P\cap A\neq\es\}$. Then $A\subset Q$, $B\subset X\setminus Q$ and $Q$ is a $CZ$-set of $X$ ($Q$ and $X\setminus Q$ are countable unions of elements of $\cP$, hence $Zer_\si$-sets).
\end{proof}

Let $X$ be a space. Disjoint sets $A,B\subset X$ are called {\em $\om$-$G_\de$-separated} if the sets $A'$ and $B'$ are $G_\de$-separated for any countable $A'\subset A$ and $B'\subset B$.

\begin{definition}
A space $X$ is called a {\em $\rho_2$-space} if for $A\subset X$, from the fact that the sets $A$ and $X\setminus A$ are $\om$-$G_\de$-separated it follows that $A$ is a $CZ$-set.
\end{definition}

\begin{definition}
A space $X$ is called a {\em strongly $\rho_2$-space} if any $\om$-$G_\de$-separated sets $A,B\subset X$ are $CZ$-separated.
\end{definition}

\begin{proposition} \label{p:sp:2}
Let $X$ be a space. The space $K_1(X,\ot)$ is realcompact if and only if $X$ is a $\rho_2$-space.
\end{proposition}
\begin{proof}

Let $K=K_1(X,\ot)\subset\ot^X$; for $A\subset X$ the indicator function $\chi_A$ belongs to $K$ iff $A$ is a $CZ$-set. The space $K$ is dense in $\ot^X$, since finite subsets of $X$ are $CZ$-separated. For $A\subset X$ and a countable $S\subset X$ there is $h\in K$ with $h\restriction S=\chi_A\restriction S$ iff the countable sets $A\cap S$ and $S\setminus A$ are $CZ$-separated, iff (by Lemma \ref{lem2}, countable sets being Lindel\"of) they are $G_\de$-separated. Hence the condition of Proposition \ref{p:dense} for $Z=K$ reads: if $A$ and $X\setminus A$ are $\om$-$G_\de$-separated, then $\chi_A\in K$, i.e. $A$ is a $CZ$-set. This is exactly the definition of a $\rho_2$-space.

\end{proof}

A space $X$ is called {\em weakly hereditarily Baire} ($whb$-space) if any closed subset $F\subset X$ can be represented as a union of two sets $F_1$ and $F_2$, where $F_1$ is a Baire space and $F_2$ is countable.

\begin{proposition} \label{p:sp:whb}
Let $X$ be a separable metrizable $whb$-space. Then the space $K_1(X,\ot)$ is realcompact.
\end{proposition}
\begin{proof}
This follows from Corollary \ref{t:rc:1} below, since $K_1(X,\ot)$ is closed in $B_1(X)$.
\end{proof}

\begin{corollary}\label{c:sp:whb}
Let $X\subset \mathbb{R}$ be a Luzin set. Then the space $K_1(X,\ot)$ is realcompact.
\end{corollary}
\begin{proof}
By Proposition \ref{p:sp:whb} it suffices to check that a Luzin set $X\subset\R$ (every nowhere dense subset of $\R$ meets $X$ in a countable set) is a $whb$-space. Let $F$ be closed in $X$. Let $C$ be the union of all countable relatively open subsets of $F$; $C$ is countable and $F'=F\setminus C$ is closed in $F$, and every nonempty relatively open subset of $F'$ is uncountable. Let $W$ be the interior in $\R$ of $\overline{F'}^{\R}$. The set $F'\setminus W$ lies in the boundary of the closed set $\overline{F'}^{\R}$, which is nowhere dense, so it is countable. Put $F_2=C\cup(F'\setminus W)$ and $F_1=F'\cap W$. The set $F_1$ is dense in the open set $W$, so a nowhere dense subset of $F_1$ is nowhere dense in $\R$ and hence countable. Therefore meager subsets of $F_1$ are countable, while nonempty open subsets of $F_1$ are uncountable; so $F_1$ is a Baire space.
\end{proof}

\begin{proposition} \label{p:sp:3}
Let $X$ be a space. If the space $K_1(X,\ot)$ is Lindel\"{o}f then $X$ is a strongly $\rho_2$-space.
\end{proposition}
\begin{proof}

Let $A,B\subset X$ be $\om$-$G_\de$-separated. For countable $A'\subset A$, $B'\subset B$ let $\Phi_{A',B'}=\{g\in K_1(X,\ot): g(A')\subset\{1\},\ g(B')\subset\{0\}\}$. These sets are closed in $K_1(X,\ot)$ and, by Lemma \ref{lem2}, nonempty (countable sets are Lindel\"of). The family $\{\Phi_{A',B'}\}$ has the countable intersection property, since a countable union of countable sets is countable. As $K_1(X,\ot)$ is Lindel\"of, there is $g$ in all $\Phi_{A',B'}$; then $Q=g^{-1}(1)$ is a $CZ$-set with $A\subset Q$ and $B\cap Q=\es$.

\end{proof}

\begin{question}\label{q:rho1}
Is every $\rho_2$-space a strongly $\rho_2$-space?
\end{question}

In a metrizable space $CZ$-sets are exactly the sets which are both $F_\si$ and $G_\de$.

\begin{lemma}\label{l:kernel}
Let $X$ be a separable metrizable space and let $A,B\subset X$ be $\om$-$G_\de$-separated sets which are not $CZ$-separated. Then there is a nonempty closed $F\subset X$ such that $A\cap F$ and $B\cap F$ are dense in $F$ and every nonempty relatively open subset of $F$ is uncountable.
\end{lemma}
\begin{proof}
For a closed $E\subset X$ let $c(E)$ be the union of all countable relatively open subsets of $E$; since $X$ is second countable, $c(E)$ is countable and relatively open in $E$. Define closed sets $F_\xi$ by $F_0=X$,
\[
F_{\xi+1}=\big(\overline{A\cap F_\xi}\cap\overline{B\cap F_\xi}\big)\setminus c(F_\xi),
\]
and $F_\la=\bigcap_{\xi<\la}F_\xi$ for limit $\la$. The sequence is decreasing, so $F_\zeta=F_{\zeta+1}$ for some $\zeta<\om_1$. If $F_\zeta\ne\es$, then $c(F_\zeta)=\es$ and $F=F_\zeta$ is as required.

Suppose $F_\zeta=\es$. For $\xi<\zeta$ let $V_\xi=c(F_\xi)$. The countable sets $A\cap V_\xi$ and $B\cap V_\xi$ are $G_\de$-separated: there are disjoint $G_\de$-sets $P_\xi\supset A\cap V_\xi$ and $R_\xi\supset B\cap V_\xi$. Since $V_\xi$ is locally closed, hence $G_\de$, the countable sets $P_\xi\cap V_\xi$ and $V_\xi\setminus P_\xi$ are both $F_\si$ and $G_\de$. Put
\[
Q=\bigcup_{\xi<\zeta}\Big(\big((F_\xi\setminus V_\xi)\setminus\overline{B\cap F_\xi}\big)\cup(P_\xi\cap V_\xi)\Big).
\]
Every point of $X$ lies in exactly one set $F_\xi\setminus F_{\xi+1}$, which is the disjoint union of $V_\xi$, $(F_\xi\setminus V_\xi)\setminus\overline{B\cap F_\xi}$ and $(F_\xi\setminus V_\xi)\cap\overline{B\cap F_\xi}\setminus\overline{A\cap F_\xi}$. A point of $A\cap(F_\xi\setminus V_\xi\setminus F_{\xi+1})$ lies in $\overline{A\cap F_\xi}$, hence not in $\overline{B\cap F_\xi}$; therefore $A\subset Q$, and clearly $Q\cap B=\es$. Both $Q$ and $X\setminus Q$ are countable unions of locally closed sets and countable sets, hence $F_\si$. Thus $Q$ is a $CZ$-set separating $A$ and $B$, a contradiction.
\end{proof}

\begin{lemma}\label{l:whb}
For a separable metrizable space $X$ the following are equivalent:
\begin{enumerate}
\item $X$ is a $whb$-space;
\item for every closed $F\subset X$ the set $F\setminus c(F)$ is a Baire space ($c(F)$ as in the proof of Lemma \ref{l:kernel});
\item every nonempty closed $F\subset X$ without nonempty countable relatively open subsets contains a nonempty relatively open Baire subspace.
\end{enumerate}
\end{lemma}
\begin{proof}
We use the following simple fact: a space containing a dense Baire subspace $D$ is a Baire space (if $U_n$ are open dense, then $U_n\cap D$ are open dense in $D$, so $\bigcap_n U_n\cap D$ is dense in $D$ and hence in the space).

(2)$\Rightarrow$(1): $F=(F\setminus c(F))\cup c(F)$ and $c(F)$ is countable.

(1)$\Rightarrow$(3): let $F=F_1\cup F_2$, $F_1$ Baire, $F_2$ countable, and let every nonempty relatively open subset of $F$ be uncountable. Then $F_1$ is dense in $F$, so $F$ is a Baire space.

(3)$\Rightarrow$(2): let $F'=F\setminus c(F)$. Then $F'$ is closed and has no nonempty countable relatively open subsets (if $F'\cap O$ is countable, $O$ open, then $F\cap O\subset (F'\cap O)\cup c(F)$ is countable, so $F\cap O\subset c(F)$). Let $V\ne\es$ be relatively open in $F'$. The closed set $\overline V$ also has no nonempty countable relatively open subsets, so by (3) it contains a nonempty relatively open Baire subspace $U$; then $U\cap V$ is a nonempty relatively open Baire subspace of $F'$. Hence the union $W$ of all relatively open Baire subspaces of $F'$ is dense in $F'$. The space $W$ is Baire, being locally Baire, and therefore $F'$ is a Baire space.
\end{proof}

\begin{theorem}\label{t:rho2}
Every separable metrizable $whb$-space is a strongly $\rho_2$-space.
\end{theorem}
\begin{proof}
Let $A,B\subset X$ be $\om$-$G_\de$-separated, and suppose they are not $CZ$-separated. Let $F$ be as in Lemma \ref{l:kernel}. By Lemma \ref{l:whb}, $F=F\setminus c(F)$ is a Baire space. Take countable sets $A'\subset A\cap F$ and $B'\subset B\cap F$ dense in $F$. There are disjoint $G_\de$-sets $P\supset A'$ and $R\supset B'$. Then $P\cap F$ and $R\cap F$ are dense $G_\de$-subsets of the Baire space $F$, so they intersect, a contradiction.
\end{proof}

\begin{corollary}\label{c:rho2}
Hereditarily Baire (in particular, Polish) separable metrizable spaces, Luzin sets $X\subset\R$, and countable spaces are strongly $\rho_2$-spaces. In particular, the segment $[0,1]$ is a strongly $\rho_2$-space.
\end{corollary}
\begin{proof}
A hereditarily Baire space is a $whb$-space (take $F_1=F$, $F_2=\es$), a countable space is a $whb$-space (take $F_1=\es$, $F_2=F$), and a Luzin set is a $whb$-space, as shown in the proof of Corollary \ref{c:sp:whb}. Apply Theorem \ref{t:rho2}.
\end{proof}

\begin{proposition}\label{p:sigma}
Let $X$ be a separable metrizable space and $X=\bigcup_{n\in\om}X_n$, where every $X_n$ is a closed strongly $\rho_2$-subspace of $X$. Then $X$ is a strongly $\rho_2$-space. In particular, every countable union of closed $whb$-subspaces (e.g., every $\si$-compact metrizable space) is a strongly $\rho_2$-space.
\end{proposition}
\begin{proof}
Let $A,B\subset X$ be $\om$-$G_\de$-separated. Then $A\cap X_n$ and $B\cap X_n$ are $\om$-$G_\de$-separated in $X_n$, so there is a set $Q_n\subset X_n$ which is $F_\si$ and $G_\de$ in $X_n$ with $A\cap X_n\subset Q_n$ and $Q_n\cap B=\es$. Put $D_n=X_n\setminus\bigcup_{m<n}X_m$ and $Q=\bigcup_n(Q_n\cap D_n)$. Since $X_n$ is closed and $D_n$ is locally closed in $X$, the sets $Q=\bigcup_n(Q_n\cap D_n)$ and $X\setminus Q=\bigcup_n\big((X_n\setminus Q_n)\cap D_n\big)$ are $F_\si$ in $X$. Thus $Q$ is a $CZ$-set with $A\subset Q$ and $Q\cap B=\es$. The last assertion follows from Theorem \ref{t:rho2}, since compact metrizable spaces are Polish.
\end{proof}

For instance, $X=\Q\times\ot^\om$ is a $\si$-compact strongly $\rho_2$-space which is not a $whb$-space: $X$ has no nonempty countable open subsets and is not a Baire space.

Recall that a separable metrizable space $X$ is a {\it $\la$-set} if every countable subset of $X$ is a $G_\de$-set in $X$, and a {\it $Q$-set} if every subset of $X$ is an $F_\si$-set in $X$.

\begin{example}\label{ex:lambda}
Let $X$ be a $\la$-set. Then:
\begin{enumerate}
\item any two disjoint subsets of $X$ are $\om$-$G_\de$-separated;
\item $X$ is a $\rho_2$-space $\iff$ $X$ is a strongly $\rho_2$-space $\iff$ $X$ is a $Q$-set;
\item $K_1(X,\ot)$ is $G_\de$-dense in $\ot^X$, hence pseudocompact, and $K_1(X,\ot)$ is realcompact if and only if $K_1(X,\ot)=\ot^X$, if and only if $X$ is a $Q$-set.
\end{enumerate}
In particular, if there is a $\la$-set which is not a $Q$-set, then there is a separable metrizable space which is not a $\rho_2$-space, and for this space $X$ the space $K_1(X,\ot)$ is pseudocompact but not realcompact.
\end{example}
\begin{proof}
(1) If $A'\subset A$ and $B'\subset B$ are countable, then $A'$ and $B'$ are disjoint $G_\de$-sets in $X$.

(2) By (1), $X$ is $\rho_2$ (strongly $\rho_2$) iff every $A\subset X$ is a $CZ$-set, i.e. both $F_\si$ and $G_\de$. If every subset of $X$ is $F_\si$, then, passing to complements, every subset is also $G_\de$.

(3) A nonempty $G_\de$-subset of $\ot^X$ contains a set $\{g: g\restriction C=\xi\}$, where $C\subset X$ is countable and $\xi\in\ot^C$. The countable set $S=\xi^{-1}(1)$ is $F_\si$ and, since $X$ is a $\la$-set, $G_\de$ in $X$; so the indicator function of $S$ belongs to $K_1(X,\ot)$ and to this set. Thus $K_1(X,\ot)$ is $G_\de$-dense in the compact space $\ot^X$ and, hence, pseudocompact. A pseudocompact realcompact space is compact, so if $K_1(X,\ot)$ is realcompact then it is closed and dense in $\ot^X$, i.e. $K_1(X,\ot)=\ot^X$; the latter means that every subset of $X$ is a $CZ$-set, i.e. $X$ is a $Q$-set (cf. \cite[Proposition 2.14]{GOR26}).
\end{proof}

\begin{proposition}\label{p:lambdaY}
Let $X$ be a $\la$-set and $Y$ a space with $|Y|\ge2$. Then:
\begin{enumerate}
\item $K_1(X,Y)$ is $G_\de$-dense in $Y^X$; in particular, if $Y$ is compact, then $K_1(X,Y)$ is pseudocompact (e.g., $K_1(X,\ot)$ and $K_1(X,[0,1])$ are pseudocompact);
\item $K_1(X,Y)$ is realcompact if and only if $X$ is a $Q$-set and $Y$ is realcompact.
\end{enumerate}
In particular, if $X$ is a $\la$-set which is not a $Q$-set, then $K_1(X,Y)$ is not realcompact for any $Y$ with $|Y|\ge2$; for instance, $B_1(X)=K_1(X,\R)$ is not realcompact.
\end{proposition}
\begin{proof}
(1) A nonempty $G_\de$-subset of $Y^X$ contains a set $\{g: g\restriction C=\xi\}$, where $C\subset X$ is countable and $\xi\in Y^C$. Since $X$ is a $\la$-set, every subset of $C$ is a countable union of points and also a $G_\de$-set in $X$, so it is a $CZ$-set; in particular, $C$ and $X\setminus C$ are $CZ$-sets. Let $\xi(C)=\{y_n:n\in\om\}$, $S_n=\xi^{-1}(y_n)$, fix $y_0\in Y$, and let $g=\xi$ on $C$ and $g=y_0$ on $X\setminus C$. For a cozero set $U\subset Y$ the set $g^{-1}(U)$ is the union of the $S_n$ with $y_n\in U$ and, possibly, of $X\setminus C$; hence it is a $Zer_\si$-set and $g\in K_1(X,Y)$. If $Y$ is compact, then $K_1(X,Y)$ is a $G_\de$-dense subspace of the compact space $Y^X$ and, hence, is pseudocompact.

(2) If $X$ is a $Q$-set, then every subset of $X$ is a $CZ$-set, so $K_1(X,Y)=Y^X$, which is realcompact if $Y$ is. Conversely, let $K_1(X,Y)$ be realcompact. The constant functions form a closed subspace of $K_1(X,Y)$ homeomorphic to $Y$, so $Y$ is realcompact. Fix $y_0\ne y_1$ in $Y$ and cozero sets $U_i\ni y_i$ with $y_{1-i}\notin U_i$. A function $f\colon X\to\{y_0,y_1\}$ belongs to $K_1(X,Y)$ iff $f^{-1}(y_1)$ is a $CZ$-set. Hence $\{f\in K_1(X,Y): f(X)\subset\{y_0,y_1\}\}$ is a closed subspace of $K_1(X,Y)$ homeomorphic to $K_1(X,\ot)$. Therefore $K_1(X,\ot)$ is realcompact and, by Example \ref{ex:lambda}(3), $X$ is a $Q$-set.
\end{proof}

\begin{remark}
Compare Proposition \ref{p:lambdaY} with Corollary \ref{t:rc:1}: an uncountable $\la$-set which is not a $Q$-set is not a $whb$-space (a countable dense subset is a dense $G_\de$-set; see also \cite[Theorem 8.47]{Bukovsky2011}), so there is no contradiction. Note also that for every $\la$-set $X$ the space $B_1(X)$ is a Choquet space \cite[Theorem 1.1(iv)]{GOR26}; thus, if $X$ is a $\la$-set which is not a $Q$-set, then $B_1(X)$ is Choquet but not realcompact.
\end{remark}

\begin{remark}\label{r:lambdaQ}
There is a $\la$-set of cardinality $\mathfrak b$ in ZFC (Rothberger, see \cite[Theorem 5.6]{Miller1984}, \cite[Theorem 8.48, Corollary 8.49]{Bukovsky2011}), and every subset of a $\la$-set is a $\la$-set. A $Q$-set $X$ satisfies $2^{|X|}\le\mathfrak c$, since $X$ has at most $\mathfrak c$ $F_\si$-subsets, in particular, $|X|<\mathfrak c$ \cite[p.~313]{Bukovsky2011}. Hence a $\la$-set which is not a $Q$-set exists in each of the following cases: $2^{\aleph_0}<2^{\aleph_1}$ (take a $\la$-set of cardinality $\aleph_1$); $\mathfrak b=\mathfrak c$ (e.g., under MA); $2^{\mathfrak b}>\mathfrak c$; there are no uncountable $Q$-sets, e.g. in the model of \cite[Theorem 22]{Miller1979}, where every uncountable separable metrizable space has Borel order $\om_1$, so there are no uncountable $\si$-sets. We do not know whether such a $\la$-set exists in ZFC; this question was asked on MathOverflow \cite{MO507641}.
\end{remark}

\begin{question}\label{q:rho2}
Is every separable metrizable space a strongly $\rho_2$-space ($\rho_2$-space) in ZFC? In particular, is there in ZFC a $\la$-set which is not a $Q$-set?
\end{question}

\subsection{Realcompactness of $B_1(X)$}

\begin{lemma}\label{l:b1cz}
Let $X$ be a space and $f\colon X\to\R$. Then $f\in B_1(X)$ if and only if for all rational $r<s$ the sets $\{f\le r\}$ and $\{f\ge s\}$ are $CZ$-separated.
\end{lemma}
\begin{proof}
If $f\in B_1(X)$, then $\{f<s\}$ is a $Zer_\si$-set and $\{f<s\}=\bigcap_n\{f\le s-\frac1n\}$ is a $Coz_\de$-set, so it separates $\{f\le r\}$ and $\{f\ge s\}$. Conversely, let $H_{r,r'}$ be a $CZ$-set with $\{f\le r\}\subset H_{r,r'}$ and $H_{r,r'}\cap\{f\ge r'\}=\es$. Then $\{f<s\}=\bigcup\{H_{r,r'}:r<r'<s\}$ is a $Zer_\si$-set; similarly $\{f>s\}$ is a $Zer_\si$-set. Hence $f\in B_1(X)$.
\end{proof}

\begin{theorem}\label{t:rc:A}
If $X$ is a strongly $\rho_2$-space, then $B_1(X)$ is realcompact.
\end{theorem}
\begin{proof}
Let $f\in\R^X\setminus B_1(X)$. By Lemma \ref{l:b1cz}, there are $r<s$ such that $\{f\le r\}$ and $\{f\ge s\}$ are not $CZ$-separated; since $X$ is a strongly $\rho_2$-space, they are not $\om$-$G_\de$-separated. So there are countable $A\subset\{f\le r\}$ and $B\subset\{f\ge s\}$ which are not $G_\de$-separated. The set
$W_f=\{g\in\R^X: g\restriction(A\cup B)=f\restriction(A\cup B)\}$ is a zero set of $\R^X$ containing $f$. If $g\in W_f\cap B_1(X)$, then $\{g\le r\}\supset A$ and $\{g\ge s\}\supset B$ are disjoint $Coz_\de$-sets, hence $G_\de$-sets, a contradiction. Thus $W_f\cap B_1(X)=\es$, and $B_1(X)=\bigcap\{\R^X\setminus W_f: f\in\R^X\setminus B_1(X)\}$ is an intersection of cozero subsets of the realcompact space $\R^X$. Cozero subsets of realcompact spaces are realcompact, and intersections of realcompact subspaces are realcompact \cite{Eng}. Hence $B_1(X)$ is realcompact.
\end{proof}

\begin{corollary}\label{t:rc:1}
Let $X$ be a separable metrizable $whb$-space (e.g., a hereditarily Baire separable metrizable space, or a Luzin set). Then $B_1(X)$ and $K_1(X,\ot)$ are realcompact.
\end{corollary}
\begin{proof}
This follows from Theorems \ref{t:rho2} and \ref{t:rc:A}, since $K_1(X,\ot)$ is closed in $B_1(X)$.
\end{proof}

\begin{remark}
If $B_1(X)$ is realcompact, then so is its closed subspace $K_1(X,\ot)$, and, by Proposition \ref{p:sp:2}, $X$ is a $\rho_2$-space. Thus
strongly $\rho_2$ $\Rightarrow$ $B_1(X)$ is realcompact $\Rightarrow$ $K_1(X,\ot)$ is realcompact $\iff$ $\rho_2$.
\end{remark}

By Proposition \ref{p:sigma} and Theorem \ref{t:rc:A}, $B_1(X)$ and $K_1(X,\ot)$ are also realcompact if $X$ is a countable union of closed $whb$-subspaces, e.g., if $X$ is a $\si$-compact metrizable space.

By Proposition \ref{p:dense}, $B_1(X)$ is realcompact if and only if every function $f\colon X\to\R$ which coincides on every countable subset of $X$ with some Baire-one function is itself Baire-one; by Lemma \ref{l:b1cz} this concerns the pairs of sets $\{f\le r\}$, $\{f\ge s\}$, whereas realcompactness of $K_1(X,\ot)$ concerns only the pairs $A$, $X\setminus A$. For separable metrizable $whb$-spaces both spaces are realcompact (Corollary \ref{t:rc:1}), and for $\la$-sets both are realcompact if and only if $X$ is a $Q$-set (Example \ref{ex:lambda} and Proposition \ref{p:lambdaY}). A positive answer to Question \ref{q:rho1} would imply a positive answer to the following question.

\begin{question}\label{q:K1B1}
Let $X$ be a space such that $K_1(X,\ot)$ is realcompact, i.e., $X$ is a $\rho_2$-space. Is $B_1(X)$ realcompact?
\end{question}

All separable metrizable spaces $X$ for which we know that $B_1(X)$ (or $K_1(X,\ot)$) is not realcompact are $\la$-sets which are not $Q$-sets (Proposition \ref{p:lambdaY}); their existence is known only under additional set-theoretic assumptions (Remark \ref{r:lambdaQ}). On the other hand, a counterexample cannot be a $whb$-space or a countable union of closed $whb$-subspaces (Corollary \ref{t:rc:1}, Proposition \ref{p:sigma}).

\begin{question}\label{q:ZFC}
Is there in ZFC a separable metrizable space $X$ such that $B_1(X)$ is not realcompact? such that $K_1(X,\ot)$ is not realcompact (i.e., $X$ is not a $\rho_2$-space, cf. Question \ref{q:rho2})?
\end{question}

Analytic spaces are natural candidates for a positive result. If $X$ is an analytic space and $A,B\subset X$ are $\om$-$G_\de$-separated but not $CZ$-separated, then the set $F$ of Lemma \ref{l:kernel} is meager in its closure $K$ in a completion of $X$. Indeed, take countable sets $A'\subset A\cap F$ and $B'\subset B\cap F$ dense in $F$ and disjoint $G_\de$-sets $P\supset A'$, $R\supset B'$ of $X$; extending them to $G_\de$-sets $\tilde P$, $\tilde R$ of $K$, we get a dense $G_\de$-subset $\tilde P\cap\tilde R$ of $K$ disjoint from $F$. Since $F$ is analytic, it has the Baire property in $K$, but we do not know whether this leads to a contradiction.

\begin{question}\label{q:analytic}
Let $X$ be a Suslin (analytic) separable metrizable space. Is $B_1(X)$ realcompact? Is $X$ a strongly $\rho_2$-space?
\end{question}

\begin{remark}
For hereditarily Baire separable metrizable spaces Corollary \ref{t:rc:1} can also be proved using fragmentability; we sketch this argument.
\end{remark}

Recall \cite{JR} that for some $\varepsilon>0$ a function $f: X\rightarrow (Y,d)$ from a topological space to a metric one is
said to be $\varepsilon$-{\it fragmented}, if for every closed nonempty set $F\subseteq X$ there exists a nonempty relatively open
set $U\subseteq F$ such that $diam f(U)<\varepsilon$. If $f$ is $\varepsilon$-fragmented for every $\varepsilon>0$, then it is called fragmented.
If $X$ is hereditarily Baire and metrizable, then every Baire-one map $f:X \rightarrow (Y, d)$ is barely continuous (i.e., for
every nonempty closed set $F\subseteq X$ the restriction $f\upharpoonright F$ has a point of continuity) and, hence, is fragmented
(see \cite{Kur}, 31.X). If $X$ is a paracompact space in which every closed set is $G_{\delta}$, then every fragmented map $f:X\rightarrow (Y,d)$ is Baire-one in the case either $\dim X=0$, or $Y$ is a contractible locally path-connected space \cite{KM}. Sketch of the alternative proof for a hereditarily Baire separable metrizable $X$: let $f\in \mathbb{R}^X\setminus B_1(X)$. By \cite{KM}, $f$ is not fragmented, so there are $\varepsilon>0$ and a closed nonempty $F\subseteq X$ such that $\operatorname{diam} f(U)\geq\varepsilon$ for any nonempty relatively open $U\subseteq F$. Using that $F$ is a Baire space one finds rationals $r<s$ and a nonempty relatively open $W\subset F$ such that $\{x\in W: f(x)<r\}$ and $\{x\in W: f(x)>s\}$ are dense in $W$; countable dense subsets $A$, $B$ of these sets give a zero set $W_f\ni f$ disjoint from $B_1(X)$, since a Baire-one function on the Baire space $\overline W$ has a point of continuity.

\begin{question}\label{q:whb}
Let $X$ be a (non-separable) $whb$-space. Is it true that the space $B_1(X)$ is realcompact? Is $K_1(X,\ot)$ realcompact, i.e., is $X$ a $\rho_2$-space? Is $X$ a strongly $\rho_2$-space?
\end{question}

\section{$nCM$-spaces}\label{sec:ncm}

A topological space $X$ is
an $nCM$-space if no continuous mapping from $X$ into $[0, 1]$ is a surjection, i.e., if
$X$ cannot be continuously mapped onto $[0, 1]$ (see p.304 in \cite{Bukovsky2011}). By Remark \ref{r:cantor}, $X$ is an $nCM$-space if and only if $f(X)$ contains no uncountable compact set for every $f\in C(X)$, i.e. iff condition (3) of Proposition \ref{pr1} holds.

\medskip

The following statement is obvious.

\begin{proposition}\label{p:ncm:1}
A continuous image of an $nCM$-space is an $nCM$-space.
\end{proposition}

\begin{proposition}[{\cite[Theorem 8.22]{Bukovsky2011}}]\label{p:ncm:2}
Any Tychonoff $nCM$-space $X$ has a clopen base, i.e., $X$ is zero-dimensional. If $X$ is a normal $nCM$-space then $\operatorname{Ind} X=0$, i.e., for any disjoint closed
sets $A,B\subseteq X$ there exists a clopen set $U$ such that $A\subseteq U$ and $U\cap B=\emptyset.$
\end{proposition}

\begin{theorem}\label{t:ncm:1}
Let $K$ be a non-scattered compact space and $X\subset K$ be an $nCM$-space. Then there is a crowded compact subset $M$ of $K\setminus X$.
\end{theorem}
\begin{proof}
Since $K$ is non-scattered, there is a continuous map $g: K\rightarrow [0,1]$ from $K$ onto $[0,1]$.
Let $p\colon[0,1]\to[0,1]^2$ be a Peano map (a continuous surjection). Then $f=p\circ g\colon K\to[0,1]^2$ is a continuous surjection. Let $\pi$ be the projection of $[0,1]^2$ onto the first factor. Since $X$ is an $nCM$-space, we have $\pi(f(X))\neq [0,1]$. Let $t\in [0,1]\setminus \pi(f(X))$. Then $\{t\}\times [0,1]\subset [0,1]^2\setminus f(X)$. The compact set $f^{-1}(\{t\}\times [0,1])\subset K\setminus X$ maps onto $[0,1]$, so it is not scattered and therefore contains a nonempty compact set without isolated points.
\end{proof}

\begin{corollary}\label{c:ncm:1}
If $X\subset\R$ is a Bernstein set then it is not an $nCM$-space. Hence,  it can be continuously mapped onto $[0,1]$.
\end{corollary}
\begin{proof}
Apply Theorem \ref{t:ncm:1} to $X\subset K=[-\infty,+\infty]$: a crowded compact $M\subset K\setminus X$ is uncountable, so $M\cap\R$ contains an uncountable compact subset of $\R$ disjoint from $X$, which is impossible for a Bernstein set.
\end{proof}

\begin{theorem}\label{t:ncm:2} A countable union of $nCM$-spaces is an $nCM$-space.
\end{theorem}
\begin{proof}
Let $X=\bigcup_{i<\om}X_i$ where $X_i$ is an  $nCM$-space for each $i<\om$.
Let $f\: X\to [0,1]$ be a continuous mapping from $X$ to $[0,1]$. By Proposition \ref{p:ncm:1}, $M_i=f(X_i)$ is an $nCM$-space for every $i<\om$.

By induction we construct a decreasing sequence $(K_n)_{n\ge -1}$ of nonempty compact subsets of $[0,1]$ without isolated points such that $K_n\cap M_n=\es$ for $n\ge 0$. Let $K_{-1}=[0,1]$. Suppose $K_{n-1}$ has been constructed. The set $S=K_{n-1}\cap M_n$ is closed in $M_n$. A closed subspace of a metrizable $nCM$-space is an $nCM$-space (by the Tietze--Urysohn theorem, as in Proposition \ref{p:ncm:1+1}), so $S$ is an $nCM$-space. Since $K_{n-1}$ is a nonempty crowded compact set, it is not scattered, and by Theorem \ref{t:ncm:1} there is a nonempty crowded compact $K_{n}\subset K_{n-1}\setminus S=K_{n-1}\setminus M_{n}$.

The set $K=\bigcap_{n<\om}K_n$ is not empty and $K\subset [0,1]\setminus f(X)$. Consequently, $f(X)\neq[0,1]$.
\end{proof}

\begin{proposition}\label{p:ncm:1+1}
Let $X$ be a normal $nCM$-space and $Y\subset X$ be an $F_\si$-set.
Then $Y$ is an $nCM$-space.
\end{proposition}
\begin{proof}
Let $Y$$=\bigcup_n F_n$, where $F_n$ is closed in $X$. We show that $F_n$ is an $nCM$-space. Assume the contrary.
Then there exists a continuous surjective map $f\: F_n\to [0,1]$. It follows from the Tietze-Urysohn theorem that there exists a continuous extension $\hat f\: X\to [0,1]$ of $f$. This contradicts the fact that $X$ is an $nCM$-space. Therefore, every $F_n$ is an $nCM$-space. It follows from Theorem \ref{t:ncm:2} that $Y$ is an $nCM$-space.
\end{proof}

Let us give examples of $nCM$-spaces. The definition of the classes of spaces mentioned in the next sentence can be found in \cite{Bukovsky2011}.

\begin{proposition}[{\cite[Corollary 8.59]{Bukovsky2011}}]\label{p:ncm:3}
Every  space $X$ with any of the properties: $X$ is a $\ga$-space, $X$ is a $QN$-space, $X$ is a $wQN$-space, $X$ is a $\Si$-space, is an $nCM$-space.
\end{proposition}

A space $X$ is called a {\em Rothberger} space if for any sequence $(\ga_n)_n$ of open covers of $X$ there exist $U_n\in\ga_n$ such that $X=\bigcup_n U_n$.

\begin{proposition}\label{p:ncm:4}
If $X$ is a Rothberger space then $X$ is an $nCM$-space.
\end{proposition}
\begin{proof}
The Rothberger property is preserved by continuous mappings and $[0,1]$ is not a Rothberger space \cite[Corollary 8.103]{Bukovsky2011}; see also \cite[Corollary 8.104]{Bukovsky2011}.
\end{proof}

\begin{proposition}\label{p:ncm:5}
{\rm $\mathrm{(CH)}$} There is a non-$nCM$-space $X\subset \R\setminus \Q$ that is concentrated on the set $\Q$. The space $Y=X\cup \Q$ is  Rothberger and hence an $nCM$-space.
\end{proposition}
\begin{proof}
Assuming the continuum hypothesis there exists a set $X\subset \R\setminus \Q$ concentrated on the rationals which can be mapped continuously onto $\ot^\om$ (\cite{R} or \cite[Theorem 9.4, p. 225]{Miller1984}), and hence onto $[0,1]$.
The space $Y=X\cup\Q$ is concentrated on $\Q=\{q_n:n\in\om\}$, so it is Rothberger  \cite[Theorem 8.105]{Bukovsky2011}.
\end{proof}

\begin{corollary}\label{c:ncm:2}
{\rm $\mathrm{(CH)}$} In the class of separable metrizable spaces, the property of being an $nCM$-space is not inherited by subspaces, even by $G_\de$-subspaces (in Proposition \ref{p:ncm:5}, $X=Y\setminus\Q$ is a $G_\de$-set in $Y$).
\end{corollary}

\begin{question}
Is there in ZFC a separable metrizable $nCM$-space containing a subspace which is not an $nCM$-space?
\end{question}

\begin{question}
Let $X$ be a space which is not an $nCM$-space. Is $e(K_1(X,\ot))\ge\mathfrak{c}$?
\end{question}

\begin{remark}\label{r:luzin}
Let $X\subset\R$ be an uncountable Luzin set (such sets exist, e.g., under CH). Then $X$ satisfies all necessary conditions for the Lindel\"of property of $B_1(X)$ known to us:
\begin{enumerate}
\item $X$ is zero-dimensional (it contains no interval), perfect, every closed subset of $X$ is a $CZ$-set, and $s(X)=\aleph_0$, cf. Theorem \ref{t:pr:1};
\item $X$ is a Rothberger space \cite[Corollary 8.106]{Bukovsky2011} (it is concentrated on every countable dense subset of itself, see also \cite[Theorem 3.1(i) and \S 9]{Miller1984}), and hence an $nCM$-space (Proposition \ref{p:ncm:4}); in particular, every continuous image of $X$ is $k$-scattered and Lemma \ref{l:cantor:2} gives no uncountable closed discrete subsets of $B_1(X)$;
\item $X$ is not a $\si$-set \cite[Corollary 8.44]{Bukovsky2011} (in fact, $\mathrm{ord}(X)=3$ \cite[Theorem 15]{Miller1979}); this agrees with Remark \ref{r:gor};
\item $B_1(X)$ and $K_1(X,\ot)$ are realcompact (Corollary \ref{t:rc:1}).
\end{enumerate}
On the other hand, $X$ is not a $\la$-set (see the proof of \cite[Theorem 15]{Miller1979}), so $B_1(X)$ is not a Choquet space \cite[Theorem 1.1(iv)]{GOR26}.
\end{remark}

\begin{question}\label{q:luzinLind}
Can $B_1(X)$ be Lindel\"of (normal, or of countable extent) for an uncountable Luzin set $X\subset\R$?
\end{question}


\section{Extent and normality}\label{sec:en}

A topological space $X$ is called {\it weakly normal over a class $\mathcal{P}$ of spaces}
if for any pair of disjoint closed subsets $A,B\subseteq X$ there exist a space $P_{A,B}\in \mathcal{P}$ and a
continuous map $f_{A,B} : X \rightarrow P_{A,B}$ such that $f_{A,B}(X)=P_{A,B}$ and $f_{A,B}(A)\cap f_{A,B}(B)=\emptyset$. If $\mathcal{P}$
 is the class consisting of all subspaces of a space $Y$ then we say that $X$ is {\it weakly
normal over} the space $Y$. In the case $Y=\mathbb{R}^{\omega}$, the space $X$ is simply called {\it weakly
normal} (Def. 1.1 in \cite{Ya}; the notion is due to Arhangel'skii, see \cite{ArhCleav}). Every normal space is weakly normal: an Urysohn function $f\colon X\to[0,1]$ separating $A$ and $B$ will do.


\begin{remark}
If in this definition one requires $\overline{f(A)}\cap\overline{f(B)}=\es$, then the notion coincides with normality (compose $f$ with a distance function). So in the proof of Theorem \ref{t:en:1} one must obtain exactly $f(F)\cap f(G)\ne\es$.
\end{remark}

\begin{theorem}\label{t:en:1}
Let $X$ be a space and $Y$ be a separable metrizable space.
If a space $K_1(X,Y)$ is weakly normal, then $e(K_1(X,Y))\leq\om$.
\end{theorem}
\begin{proof}
 If $X$ is countable, then $K_1(X,Y)\subset Y^X$ is second countable and $e(K_1(X,Y))\le\om$; if $|Y|\le1$ there is nothing to prove. So let $X$ be uncountable and $|Y|\ge2$.
Assume the contrary. Let $\{f_\al:\al<\om_1\}$ be a discrete closed subset of $K_1(X,Y)$, $f_\al\ne f_\beta$ for $\al\ne\beta$.
By Proposition  \ref{p:en:1}, there is a disjoint uncountable family of zero-subsets $\{ F_\al:\al<\om_1\}$ of $X$. By Lemma \ref{l:fact}, there is a nonempty zero-set $G_\al\subset F_\al$ such that $\restr{f_\al}{G_\al}$ is constant. Let $g_\al$ be a function on $X$ such that $\restr {f_\al}{(X\setminus G_\al)}=\restr {g_\al}{(X\setminus G_\al)}$,
$\restr{g_\al}{G_\al}$ is a constant function and  $\restr{g_\al}{G_\al}\neq\restr{f_\al}{G_\al}$. By Lemma \ref{l:glue}, $g_\al\in K_1(X,Y)$.

Fix a metric $d$ on $Y$ and, for $p\in Y^X$, a finite $K\subset X$ and $\de>0$, let $[p,K,\de]=\{q\in Y^X: d(q(x),p(x))<\de$ for $x\in K\}$. Note that $g_\al\restriction(X\setminus G_\al)=f_\al\restriction(X\setminus G_\al)$, $g_\al\ne f_\al$, and every finite subset of $X$ meets only finitely many $G_\al$.

We will use the following simple remark: for every finite $E\subset X$ and every $\xi\in Y^E$ there is $u\in K_1(X,Y)$ with $u\restriction E=\xi$. Indeed, take pairwise disjoint cozero sets $V_e\ni e$, $e\in E$, and put $u=\xi(e)$ on $V_e$ and $u=y_0$ on $X\setminus\bigcup_e V_e$ for some $y_0\in Y$; by Lemma \ref{l:glue}, $u\in K_1(X,Y)$.

{\bf Claim 1.} For every continuous map $h$ of $K_1(X,Y)$ into a separable metrizable space $M$ the set $\{\al<\om_1: h(f_\al)\neq h(g_\al)\}$ is countable.

\begin{proof}[Proof of Claim 1]
Fix a metric $\rho$ on $M$. Assume the contrary; then there are $\varepsilon>0$ and an uncountable $B_0\subset\om_1$ such that $\rho(h(f_\al),h(g_\al))\ge\varepsilon$ for $\al\in B_0$. For $\al\in B_0$ choose finite $K_\al,L_\al\subset X$ and $\de_\al>0$ such that $\rho(h(u),h(f_\al))<\varepsilon/3$ for $u\in[f_\al,K_\al,\de_\al]\cap K_1(X,Y)$ and $\rho(h(u),h(g_\al))<\varepsilon/3$ for $u\in[g_\al,L_\al,\de_\al]\cap K_1(X,Y)$.
By the $\De$-system lemma there are an uncountable $B_1\subset B_0$, a finite $R\subset X$ and $\de>0$ such that $\{K_\al\cup L_\al:\al\in B_1\}$ is a $\De$-system with root $R$ and $\de_\al\ge\de$ for $\al\in B_1$. Since $R$ meets only finitely many $G_\al$, we may assume that $R\cap G_\al=\es$, so $g_\al\restriction R=f_\al\restriction R$, for $\al\in B_1$. Since $M$ and $Y^R$ are separable metrizable, there is an uncountable $B\subset B_1$ such that the sets $\{h(f_\al):\al\in B\}$ and $\{h(g_\al):\al\in B\}$ have diameter $<\varepsilon/6$ and $d(f_\al(x),f_\beta(x))<\de$ for $\al,\beta\in B$ and $x\in R$.

Take $\al\ne\beta$ in $B$. Since $K_\al\cap L_\beta\subset R$, there is $u\in K_1(X,Y)$ with $u\restriction K_\al=f_\al\restriction K_\al$ and $u\restriction(L_\beta\setminus K_\al)=g_\beta\restriction(L_\beta\setminus K_\al)$. For $x\in K_\al\cap L_\beta\subset R$ we have $d(u(x),g_\beta(x))=d(f_\al(x),f_\beta(x))<\de\le\de_\beta$. Hence $u\in[f_\al,K_\al,\de_\al]\cap[g_\beta,L_\beta,\de_\beta]$, and
\[
\rho(h(f_\beta),h(g_\beta))\le \rho(h(f_\beta),h(f_\al))+\rho(h(f_\al),h(u))+\rho(h(u),h(g_\beta))<\varepsilon/6+\varepsilon/3+\varepsilon/3<\varepsilon,
\]
a contradiction.
\end{proof}

We say that $p\in Y^X$ is a {\it limit point} of an indexed family $\{p_\al:\al<\om_1\}\subset Y^X$ if every neighbourhood of $p$ contains $p_\al$ for infinitely many $\al$.

{\bf Claim 2.} The families $\{f_\al:\al<\om_1\}$ and $\{g_\al:\al<\om_1\}$ have the same limit points in~$Y^X$.

\begin{proof}[Proof of Claim 2]
Let $p\in Y^X$ and $W=[p,K,\de]$. The finite set $K$ meets $G_\al$ only for $\al$ in a finite set $S$, and for $\al\notin S$ we have $f_\al\restriction K=g_\al\restriction K$, so $f_\al\in W\iff g_\al\in W$. Hence $W$ contains $f_\al$ for infinitely many $\al$ iff it contains $g_\al$ for infinitely many $\al$.
\end{proof}

Since $\{f_\al\}$ is a closed discrete set of pairwise distinct functions, no point of $K_1(X,Y)$ is a limit point of $\{f_\al\}$. By Claim 2, no point of $K_1(X,Y)$ is a limit point of $\{g_\al\}$. Consequently, $\{g_\al:\al<\om_1\}$ is a closed discrete subset of $K_1(X,Y)$, and for every $q\in K_1(X,Y)$ the set $\{\al: g_\al=q\}$ is finite (otherwise $q$ would be a limit point of $\{g_\al\}$). Define
\[
\psi\colon\om_1\to\om_1+1,\quad \psi(\al)=\begin{cases}
\beta, & \text{if } g_\al=f_\beta \text{ for some } \beta\in\om_1,\\
\om_1, & \text{otherwise.}
\end{cases}
\]
This is well defined, since the $f_\beta$ are pairwise distinct, and $\psi(\al)\ne\al$, since $g_\al\ne f_\al$.

{\bf Claim 3.} $\psi^{-1}(\beta)$ is finite for every $\beta\in\om_1$.

\begin{proof}[Proof of Claim 3]
$\psi^{-1}(\beta)=\{\al: g_\al=f_\beta\}$, which is finite by the above.
\end{proof}

By Claim 3, the preimage under $\psi$ of a countable subset of $\om_1$ is countable. Choose by induction $\al_\xi<\om_1$, $\xi<\om_1$, such that
\[
\al_\xi\notin\{\al_\eta:\eta<\xi\}\cup\{\psi(\al_\eta):\eta<\xi\}\cup\psi^{-1}(\{\al_\eta:\eta<\xi\});
\]
this is possible since the excluded set is countable. Let $A=\{\al_\xi:\xi<\om_1\}$. Then $\psi(\al)\notin A$ for $\al\in A$. Hence the sets $F=\{f_\al:\al\in A\}$ and $G=\{g_\al:\al\in A\}$ are disjoint; they are closed in $K_1(X,Y)$ as subsets of closed discrete sets.

Since $K_1(X,Y)$ is weakly normal, there is a continuous map $h$ of $K_1(X,Y)$ onto a subspace $M$ of $\R^\om$ such that $h(F)\cap h(G)=\es$. In particular, $h(f_\al)\ne h(g_\al)$ for all $\al\in A$, which contradicts Claim 1, since $A$ is uncountable.

\end{proof}

\begin{corollary} \label{c:enba:1}
If $K_1(X,\ot)$ is normal then $e(K_1(X,\ot))\leq \om$.
\end{corollary}

\begin{proposition} \label{c:enba:2}
If $B_1(X)$ is normal then $e(B_1(X))\leq \om$.
\end{proposition}
\begin{proof} Let $B_1(X)$ be a normal space. Since $B_1(X)$ is homeomorphic to $B_1(X,(0,1))$ and $B_1(X,(0,1))$ is a convex everywhere dense subset of $[0,1]^X$, $B_1(X)$ is collectionwise normal by Reznichenko's theorem \cite{rezn90} (see also \cite[Theorem I.5.8]{arch}).
Since $B_1(X)$ is dense in $\R^X$, its Souslin number is countable, and by Lemma I.5.11(b) in \cite{arch}, $e(B_1(X))\leq \om$.
\end{proof}
Note that Proposition \ref{c:enba:2} does not depend on Theorem \ref{t:en:1}; its proof is the same as that of Theorem I.5.12 in \cite{arch} for $C_p(X)$.

\begin{proposition}\label{p:en:3}
If $e(B_1(X))< \mathfrak c$ then $X$ is an $nCM$-space.
\end{proposition}

\begin{proof} If $X$ is not an $nCM$-space, then $f(X)=[0,1]$ for some $f\in C(X)$, and $e(B_1(X))\ge\mathfrak c$ by Lemma \ref{l:cantor:2}.
\end{proof}

\begin{corollary}\label{c:norm:kscat}
If $B_1(X)$ is normal, then $X$ is an $nCM$-space (by Propositions \ref{c:enba:2} and \ref{p:en:3}); in particular every continuous image of $X$ is $k$-scattered {\rm (see the proof of Theorem \ref{t:pr:1}(4))}.
\end{corollary}

For a cardinal $\tau$, let $L(\tau)$ denote the one-point Lindel\"ofication of a discrete set of cardinality $\tau$. We have $L(\tau)=\tau+1=\tau\cup\{\tau\}$, $\tau$ is the only non-isolated point of $L(\tau)$, and neighborhoods of $\tau$ are the sets of the form $L(\tau)\setminus C$, where $C\subset \tau$, $|C|\leq\om$.

\begin{example}\label{exa:norm:1}
Let $X$ be the set of all ordinal numbers $\leq \omega_1$ with the order topology. Then $B_1(X)$ is not Lindel\"{o}f, but it is a normal space and $e(B_1(X))\leq \om$.
\end{example}
\begin{proof}
Note that $B_1(X)=C_p(L({\omega_1}))$ (Corollary 1.3 in \cite{ggk}).
It is widely known that $C_p(L({\omega_1}))$ is homeomorphic to the $\Sigma$-product $\Sigma=\{f\in \R^{\omega_1}: |\omega_1\setminus f^{-1}(0)|\leq\om\}$ in $\R^{\omega_1}$: the mapping $C_p(L({\omega_1}))\to \Sigma\times \R$, $f\mapsto ((f-f(\omega_1))\restriction\omega_1,f(\omega_1))$ is a homeomorphism and, obviously, $\Sigma\times \R$ is homeomorphic to $\Sigma$.
The $\Sigma$-product $\Sigma$ is normal and is not Lindel\"{o}f (see \cite{arch,corson}).
That the extent of $\Sigma$ is countable is well known; it also follows from Proposition \ref{c:enba:2}.
\end{proof}

The following statement was proved by V.V.~Uspenskii in his dissertation \cite{Usp-disser} and was not published in print.

\begin{example}\label{exa:norm:2}
Let $X$ be the set of all countable ordinal numbers with the order topology. Then $B_1(X)$ is not normal and $e(B_1(X))\leq \om$.
\end{example}
\begin{proof}
Every Baire-one function on $\om_1$ or on $\om_1+1$ is eventually constant (a $Zer_\si$-set is either bounded or contains a final segment), and conversely. Hence
$B_1(X)=C_p(\omega_1|L({\omega_1}))$, where $C_p(\om_1|L(\om_1))=\{f\restriction\om_1: f\in C_p(L(\om_1))\}\subset\R^{\om_1}$.
The restriction map $B_1(\om_1+1)\to B_1(\om_1)$ is a continuous bijection, and the preimage of a closed discrete set under a continuous bijection is closed and discrete. Hence, by Example \ref{exa:norm:1}, $e(B_1(X))\le e(B_1(\om_1+1))\leq \om$.

Consider a closed subset $Y=C_p(\omega_1|L({\omega_1}),\ot)$ of $C_p(\omega_1|L({\omega_1}))$. We claim that $Y$ is not normal.
Note that $Y$ is homeomorphic to $\Sigma(0)\cup \Sigma(1)$ where $\Sigma(i)$ is the $\Sigma$-product of $\{D_{\alpha}: D_{\alpha}=\ot, \alpha<\omega_1\}$ about the point ${\bf i}\in \ot^{\omega_1}$. Let $\omega_1=X_1\sqcup X_2$ be  a partition of $\omega_1$ such that $|X_1|=|X_2|=\omega_1$. Consider $F_1=\{f\in Y: f(x)=0$ for any $x\in X_1\}$ and $F_2=\{f\in Y: f(x)=1$ for any $x\in X_2\}$.
Then $F_1\subset \Sigma(0)$ and $F_2\subset \Sigma(1)$ and $F_1$ and $F_2$ are disjoint closed subsets of $Y$.

Suppose that $\varphi\colon Y\to[0,1]$ is continuous, $\varphi(F_1)=\{0\}$ and $\varphi(F_2)=\{1\}$. Every continuous function on the $\Si$-product $\Si(0)$ depends on countably many coordinates (\cite{Eng}, 2.7.14 (a)): there is a countable $C\subset\om_1$ such that $\varphi(x)=\varphi(x_C)$ for $x\in\Si(0)$, where $x_C=x$ on $C$ and $x_C=0$ on $\om_1\setminus C$.
Let $v\in\ot^{\om_1}$, $v=0$ on $C\cap X_1$ and $v=1$ elsewhere. Then $v\in F_2$. For finite $K\subset\om_1$ let $v_K=v$ on $K$ and $v_K=0$ off $K$; then $v_K\in\Si(0)$ and $v_K\to v$, so $\varphi(v)=\lim_K\varphi(v_K)=\lim_K\varphi((v_K)_C)=\varphi(v_C)$, since $(v_K)_C\to v_C$ in $\Si(0)$. But $u=v_C$ vanishes on $X_1$, so $u\in F_1$ and $\varphi(u)=0\ne1=\varphi(v)$, a contradiction. Thus $F_1$ and $F_2$ are not functionally separated, and $Y$, hence $B_1(X)$, is not normal.
\end{proof}

\begin{remark}\label{r:teng}
Let $X$ be a fortissimo space of cardinality $\om_1$ with the particular point $p$ and $Y=X\setminus\{p\}$. Then $X$ is homeomorphic to $L(\om_1)$ and $C_p(Y|X)=C_p(\om_1|L(\om_1))=B_1(\om_1)$, which is not normal by Example \ref{exa:norm:2}. This answers Teng's question (L41) in the negative.
The $\Sigma$-product of any family of real lines is normal \cite{corson} (see also \cite{KomMal}). Since $C_p(\omega_1|L({\omega_1}))$ is not normal, it is not homeomorphic to the $\Sigma$-product of $\om_1$-many real lines. Thus the answer to Teng's question (L42) is also negative.
The same argument, with $\om_1$ replaced by an arbitrary uncountable cardinal $\kappa$, shows that $C_p(Y|X)$ is not normal for a fortissimo space $X$ of any uncountable cardinality: $C_p(Y|X)$ consists of the functions on $\kappa$ that are constant outside a countable set, its closed subspace of $\ot$-valued functions is $\Si(0)\cup\Si(1)\subset\ot^\kappa$, and the proof of Example \ref{exa:norm:2} goes through verbatim.
\end{remark}

By Corollary \ref{c:norm:kscat}, if $B_1(X)$ is normal then $X$ is $k$-scattered.

Recall that a space $X$ is {\it $K$-analytic} if $X$ is a continuous image of a Lindel\"{o}f \v{C}ech-complete space.

By \cite{BBT}, a (Tychonoff) $K$-analytic space $X$ is $k$-scattered if and only if no continuous map $f:X\rightarrow [0,1]$ is surjective if and only if $X$ is functionally countable, i.e. the image $f(X)$ is countable for every $f\in C(X,\mathbb{R})$.

By Pestryakov's result (Proposition \ref{pr1} (5)) we get the following

\begin{corollary}
For a $K$-analytic space $X$ the space $B_1(X)$ is Lindel\"{o}f if and only if $X$ is countable.
\end{corollary}
\begin{proof}
If $X$ is countable, then every point is a zero set, every subset of $X$ is $Zer_\si$, and $B_1(X)=\R^X$ is Lindel\"of. If $B_1(X)$ is Lindel\"of, then $X$ is an $nCM$-space by Proposition \ref{pr1}(3), hence functionally countable by \cite{BBT}, hence countable by Proposition \ref{pr1}(5).
\end{proof}

\begin{theorem}\label{t:sp:n:1}
Let $K$ be a Cantor discontinuum. Then $e(K_1(K, \ot))=\mathfrak{c}$.
\end{theorem}
\begin{proof}

Since every $CZ$-subset of $K$ is Borel, $|K_1(K,\ot)|\le\mathfrak c$. We identify $K$ with $\ot^\om\times\ot^\om$. For $y\in\ot^\om$ define $G_y\colon K\to\ot$ by $G_y(y,t)=0$ and $G_y(x,t)=t(\De(x,y))$ for $x\ne y$. Then $G_y^{-1}(1)=\bigcup_n\{x:\De(x,y)=n\}\times\{t:t(n)=1\}$ is open, so $G_y\in K_1(K,\ot)$. Let $\bar 1\in\ot^\om$ be the constant $1$. Since $G_y(y,\bar 1)=0$ and $G_z(y,\bar1)=1$ for $z\ne y$, the set $D=\{G_y:y\in\ot^\om\}$ is discrete and $|D|=\mathfrak c$.

Let $g\in K_1(K,\ot)$ and $G_{y_\la}\to g$. We may assume $y_\la\to y^*$. If $y_\la=y^*$ cofinally, then $g=G_{y^*}$. Otherwise, passing to a subnet, $y_\la\ne y^*$ and $n_\la=\De(y^*,y_\la)\to\infty$, and $g(y^*,t)=\lim_\la t(n_\la)$ for all $t$. Then $\U=\{A\subset\om:n_\la\in A$ eventually$\}$ is a nonprincipal ultrafilter, and $t\mapsto g(y^*,t)$ is the indicator of $\U\subset\ot^\om$ (subsets of $\om$ identified with their indicators). As $g\restriction(\{y^*\}\times\ot^\om)$ is Baire-one, $\U$ would be a Borel subset of $\ot^\om$; but a nonprincipal ultrafilter does not have the Baire property (Sierpi\'nski, see \cite{Kechris}). Hence $D$ is closed in $K_1(K,\ot)$, and $e(K_1(K,\ot))=\mathfrak c$.

\end{proof}

\begin{theorem}\label{t:bern}
Let $X$ be a Bernstein set. Then $e(B_1(X))=\mathfrak c$ and $B_1(X)$ is not normal.
\end{theorem}
\begin{proof}
Corollary \ref{c:ncm:1} implies that $X$ is not an $nCM$-space.
Proposition \ref{p:en:3} implies that $e(B_1(X))\geq \mathfrak c$. On the other hand, $X$ is separable metrizable, so $X$ has at most $\mathfrak c$ open sets and hence at most $\mathfrak c$ $F_\si$-subsets; a Baire-one function is determined by the preimages of rational open intervals, which are $F_\si$-sets, so $|B_1(X)|\le\mathfrak c$ and $e(B_1(X))\le\mathfrak c$.
Proposition \ref{c:enba:2} implies that $B_1(X)$ is not normal.
\end{proof}

Since any Bernstein set of the real line is a hereditarily Baire space   (a non-hereditarily Baire separable metrizable space contains a closed copy $Q$ of the rationals (Hurewicz); then $\overline Q^{\R}\setminus Q$ is an uncountable $G_\de$-subset of $\R$ disjoint from $X$, which is impossible), we get the following corollary.

\begin{corollary}\label{c:bern}
For any  Bernstein set $X\subset \R$ the space $B_1(X)$ is realcompact, but not normal.
\end{corollary}



\end{document}